\documentclass[12pt,reqno]{amsart}

\usepackage{amsfonts, amsthm, amsmath}
\allowdisplaybreaks[4]

\usepackage{rotating}

\usepackage{tikz}

\usepackage{graphics}

\usepackage{amssymb}

\usepackage{amscd}

\usepackage[latin2]{inputenc}

\usepackage{t1enc}

\usepackage[mathscr]{eucal}

\usepackage{indentfirst}

\usepackage{graphicx}

\usepackage{graphics}

\usepackage{pict2e}

\usepackage{mathrsfs}

\usepackage{enumerate}
\usepackage[pagebackref]{hyperref}
\hypersetup{backref, pagebackref, colorlinks=true}
\usepackage{cite}
\usepackage{color}
\usepackage{epic}
\usepackage{hyperref} 
\numberwithin{equation}{section}
\theoremstyle{plain}

\newtheorem{theorem}{Theorem}[section]

\newtheorem{lemma}[theorem]{Lemma}

\newtheorem{corollary}[theorem]{Corollary}

\theoremstyle{definition}

\newtheorem{?}[theorem]{Problem}

\def\boxit#1{\leavevmode\hbox{\vrule\vtop{\vbox{\kern.33333pt\hrule
    \kern1pt\hbox{\kern1pt\vbox{#1}\kern1pt}}\kern1pt\hrule}\vrule}}

\usepackage{collectbox}



\newcommand{\f}[1]{\ifthenelse{\equal{#1}{1}}{(q;q)_\infty}{(q^{#1};q^{#1})_{\infty}}}

\begin{document}
\title[Congruences modulo powers of 5]{Congruences modulo powers of 5 for $k$-colored partitions}

\author[D. Tang]{Dazhao Tang}

\address[Dazhao Tang]{College of Mathematics and Statistics, Chongqing University, Huxi Campus LD206, Chongqing 401331, P.R. China}
\email{dazhaotang@sina.com}

\date{\today}

\begin{abstract}
Let $p_{-k}(n)$ enumerate the number of $k$-colored partitions of $n$. In this paper, we establish some infinite families of congruences modulo 25 for $k$-colored partitions. Furthermore, we prove some infinite families of Ramanujan-type congruences modulo powers of 5 for $p_{-k}(n)$ with $k=2, 6$, and $7$. For example, for all integers $n\geq0$ and $\alpha\geq1$, we prove that
\begin{align*}
p_{-2}\left(5^{2\alpha-1}n+\dfrac{7\times5^{2\alpha-1}+1}{12}\right) &\equiv0\pmod{5^{\alpha}}
\end{align*}
and
\begin{align*}
p_{-2}\left(5^{2\alpha}n+\dfrac{11\times5^{2\alpha}+1}{12}\right) &\equiv0\pmod{5^{\alpha+1}}.
\end{align*}
\end{abstract}



\maketitle

\section{Introduction}
A \emph{partition} \cite{Andr1976} of a positive integer $n$ is a finite non-increasing sequence of positive integers $\lambda_{1}\geq\lambda_{2}\cdots\geq\lambda_{r}>0$ such that $\sum_{i=1}^{r}\lambda_{i}=n$. The $\lambda_{i}$'s are called the \emph{parts} of the partition. Let $p(n)$ denote the number of partitions of $n$, then
\begin{align*}
\sum_{n=0}^{\infty}p(n)q^{n}=\dfrac{1}{(q;q)_{\infty}}.
\end{align*}

Here and throughout the paper, we adopt the following customary notation on partitions and $q$-series:
\begin{align*}
(a;q)_{\infty}=\prod_{n=0}^{\infty}(1-aq^{n}),\quad |q|<1.
\end{align*}

A partition is called a $k$-colored partition if each part can appear as $k$ colors. Let $p_{-k}(n)$ count the number of $k$-colored partitions of $n$. The generating function of $p_{-k}(n)$ is given by
\begin{align*}
\sum_{n=0}^{\infty}p_{-k}(n)q^{n} &=\dfrac{1}{(q;q)_{\infty}^{k}}.
\end{align*}

For convention, we denote $p_{-1}(n)=p(n)$.

Many congruences modulo 5 and 25 enjoyed by $p_{-k}(n)$ have been found. For example, Ramanujan's so-called ``most beautiful identity'' for partition function $p(n)$ is given by
\begin{align}\label{Ramanujan most beautiful identity}
\sum_{n=0}^{\infty}p(5n+4)q^{n}=5\dfrac{(q^{5};q^{5})_{\infty}^{5}}{(q;q)_{\infty}^{6}},
\end{align}
which readily implies one of his three classical partition congruences, namely,
\begin{align}
p(5n+4)\equiv0\pmod{5}.\label{Rama mod 5}
\end{align}

Further, we have, modulo 25,
\begin{align*}
\sum_{n=0}^{\infty}p(5n+4)q^{n} &=5\dfrac{(q^{5};q^{5})_{\infty}^{5}}{(q;q)_{\infty}^{6}}\equiv5\dfrac{(q^{5};q^{5})_{\infty}^{4}}{(q;q)_{\infty}}\\
 &=5(q^{5};q^{5})_{\infty}^{4}\sum_{n=0}^{\infty}p(n)q^{n},
\end{align*}
from which it follows easily from \eqref{Rama mod 5} that
\begin{align*}
p(25n+24)\equiv0\pmod{25}.
\end{align*}

For $k=2$, Hammond and Lewis \cite{HL2004} as well as Baruah and Sarmah \cite[Eq. (5.4)]{BS2013} proved that
\begin{align}\label{partition pairs mod 5}
p_{-2}(5n+2)\equiv p_{-2}(5n+3)\equiv p_{-2}(5n+4)\equiv0\pmod{5}.
\end{align}

Later on, Chen et al.\cite[Eq. (1.17)]{CDHS2014} proved the following congruence modulo 25 for $p_{-2}(n)$ via modular forms:
\begin{align*}
p_{-2}(25n+23)\equiv0\pmod{25}.
\end{align*}

Quite recently, with the help of modular forms, Lazarev et al. \cite{LMRS2017} provided a criterion which can be used for searching for congruences of $k$-colored partitions. As applications, they obtained that
\begin{align}
p_{-(25r+1)}(25n+24)\equiv p_{-(25r+6)}(25n+19)\equiv p_{-(25r+11)}(25n+14)\equiv0\pmod{25}.\label{LMRS:mod 25 family}
\end{align}

Following the work of Chen et al. as well as Lazarev et al., and relating to our recent work on $k$-colored partitions \cite{FT2016, FT2017}, we further consider arithmetic properties for $k$-colored partitions.

In this paper, we give an elementary proof of \eqref{LMRS:mod 25 family} and also find the following new congruences for $p_{-k}(n)$ modulo 25 for $k\equiv2,7,17\pmod{25}$.
\begin{align}
p_{-(25r+2)}(25n+23)\equiv p_{-(25r+7)}(25n+18)\equiv p_{-(25r+17)}(25n+8)\equiv0\pmod{25}.\label{Our:mod 25 family}
\end{align}

Interestingly, all six infinite families of congruences \eqref{LMRS:mod 25 family}--\eqref{Our:mod 25 family} can be written as a combined result.
\begin{theorem}\label{thm:mod25}
If $r\geq0$ and $k\in\{1,2,6,7,11,17\}$, then for any non-negative integer $n$, we have
\begin{align}
p_{-(25r+k)}(25n+25-k)\equiv0\pmod{25}.\label{mod 25 family}
\end{align}
\end{theorem}

In 1919, Ramanujan \cite{Ram1927} conjectured that for any integer $\alpha\geq1$,
\begin{align*}
p(5^{\alpha}n+\delta_{\alpha})\equiv0\pmod{5^{\alpha}},
\end{align*}
where $\delta_{\alpha}$ is the reciprocal modulo $5^{\alpha}$ of 24, and this was first proved by Watson \cite{Wat1934}.

Following the strategy of Hirschhorn \cite{HH1981, Hir2017} as well as Garvan \cite{Gar1984}, we will obtain many infinite families of congruences modulo any powers of 5 for $k$-colored partition functions $p_{-k}(n)$ with $k=2,6$, and 7.
\begin{theorem}\label{beau power}
For all integers $n\geq0$ and $\alpha\geq1$, we have
\begin{align}
p_{-2}\left(5^{2\alpha-1}n+\dfrac{7\times5^{2\alpha-1}+1}{12}\right) &\equiv0\pmod{5^{\alpha}},\label{5 power cong:odd}\\
p_{-2}\left(5^{2\alpha}n+\dfrac{11\times5^{2\alpha}+1}{12}\right) &\equiv0\pmod{5^{\alpha+1}}.\label{5 power cong:even}
\end{align}
\end{theorem}

\begin{theorem}\label{thm:cong high power}
For all integers $n\geq0$ and $\alpha\geq1$, we have
\begin{align}
p_{-6}\left(5^{\alpha}n+\dfrac{3\times5^{\alpha}+1}{4}\right) &\equiv0\pmod{5^{\alpha}},\label{5 power cong 1}\\
p_{-6}\left(5^{\alpha+1}n+\dfrac{11\times5^{\alpha}+1}{4}\right) &\equiv0\pmod{5^{\alpha+1}},\label{5 power cong 2}\\
p_{-6}\left(5^{\alpha+1}n+\dfrac{19\times5^{\alpha}+1}{4}\right) &\equiv0\pmod{5^{\alpha+1}}.\label{5 power cong 3}
\end{align}
\end{theorem}

\begin{theorem}\label{thm case 7:modulo powers of 5}
For all integers $n\geq0$ and $\alpha\geq1$, we have
\begin{align}
p_{-7}\left(5^{2\alpha-1}n+\dfrac{13\times5^{2\alpha-1}+7}{24}\right) &\equiv0\pmod{5^{\alpha}},\label{5 power cong:odd 5n+2}\\
p_{-7}\left(5^{2\alpha}n+\dfrac{17\times5^{2\alpha}+7}{24}\right) &\equiv0\pmod{5^{\alpha+1}},\label{5 power cong:5n+3}\\
p_{-7}\left(5^{2\alpha}n+\dfrac{61\times5^{2\alpha-1}+7}{24}\right) &\equiv0\pmod{5^{\alpha+1}},\label{5 power cong:5n+2}\\
p_{-7}\left(5^{2\alpha}n+\dfrac{109\times5^{2\alpha-1}+7}{24}\right) &\equiv0\pmod{5^{\alpha+1}}.\label{5 power cong:5n+4}
\end{align}
\end{theorem}

The rest of the paper is organized as follows. In Section \ref{lemmas section}, we present the background material on the $H$ operator and some useful lemmas. In Section \ref{sec:elementary pf}, we will provide an elementary proof for Theorem \ref{thm:mod25}, and in Section \ref{mod5 power} we will prove Theorems \ref{beau power}--\ref{thm case 7:modulo powers of 5}.

\section{Preliminary results}\label{lemmas section}
To prove the main results of this paper, we introduce some useful notations and terminology on $q$-series.

For convenience, we denote that
\begin{align*}
E_{j} :=(q^{j};q^{j})_{\infty}.
\end{align*}
Denote
\begin{align*}
R(q)=\dfrac{(q;q^{5})_{\infty}(q^{4};q^{5})_{\infty}}{(q^{2};q^{5})_{\infty}(q^{3};q^{5})_{\infty}}.
\end{align*}
From \cite[Eqs. (8.4.1), (8.4.2) and (8.4.4)]{Hir2017}, we see that
\begin{align}
\dfrac{E_{1}}{E_{25}} &=\dfrac{1}{R(q^{5})}-q-q^{2}R(q^{5}),\label{Rama cont frac}\\
\dfrac{E_{5}^{6}}{E_{25}^{6}} &=\dfrac{1}{R(q^{5})^{5}}-11q^{5}-q^{10}R(q^{5})^{5},\label{Rama cont frac 5}\\
\dfrac{1}{E_{1}} &=\dfrac{E_{25}^{5}}{E_{5}^{6}}\Bigg(\dfrac{1}{R(q^{5})^{4}}+\dfrac{q}{R(q^{5})^{3}}+\dfrac{2q^{2}}{R(q^{5})^{2}}+\dfrac{3q^{3}}{R(q^{5})}+5q^{4}\notag\\
 &\quad-3q^{5}R(q^{5})+2q^{6}R(q^{5})^{2}-q^{7}R(q^{5})^{3}+q^{8}R(q^{5})^{4}\Bigg).\label{Hir dissection}
\end{align}

Define
\begin{align*}
\zeta=\dfrac{E_{1}}{qE_{25}}, \quad T=\dfrac{E_{5}^{6}}{q^{5}E_{25}^{6}}.
\end{align*}

Following Hirschhorn \cite{HH1981, Hir2017}, we introduce a ``huffing'' operator $H$, which operates on a series by picking out those terms of the form $q^{5n}$, and huffing the rest away. That is,
\begin{align*}
H\left(\sum_{n=0}^{\infty}a_{n}q^{n}\right)=\sum_{n=0}^{\infty}a_{5n}q^{5n}.
\end{align*}
As in Hirschhorn \cite{Hir2017}, we define an infinite matrix $\{m(i,j)\}_{i,j\geq1}$ by
\begin{align*}
\begin{pmatrix}
5 &0 &0 &0 &0 &0 &\cdots\\ 2\times5 &5^{3} &0 &0 &0 &0 &\cdots\\ 9 &3\times5^{3} &5^{5} &0 &0 &0 &\cdots\\ 4 &22\times5^{2} &4\times5^{5} &5^{7} &0 &0 &\cdots\\ 1 &4\times5^{3} &8\times5^{5} &5^{8} &5^{9} &0 &\cdots\\ \vdots &\vdots &\vdots &\vdots &\vdots &\vdots &\ddots
\end{pmatrix}
\end{align*}
and for $i\geq6$, $m(i,1)=0$, and for $j\geq2$,
\begin{align}
m(i,j) &=25m(i-1,j-1)+25m(i-2,j-1)+15m(i-3,j-1)\notag\\
 &\quad+5m(i-4,j-1)+m(i-5,j-1).\label{recu formula}
\end{align}
By induction, it follows immediately that
\begin{lemma}[\cite{Hir2017}]
We have
\begin{itemize}
\item $m(i,j)=0$ for $j>i$.
\item $m(i,j)=0$ for $i>5j$.
\end{itemize}
\end{lemma}\label{zero vaules}

The following lemma is important for our proof.
\begin{lemma}[Eq. (6.4.9), \cite{Hir2017}]\label{key lemma}
For $j\geq1$, we have
\begin{align*}
H\left(\dfrac{1}{\zeta^{i}}\right)=\sum_{j=1}^{\infty}\dfrac{m(i,j)}{T^{j}}=\sum_{j=1}^{i}\dfrac{m(i,j)}{T^{j}}.
\end{align*}
\end{lemma}

Employing the binomial theorem, we can easily establish the following congruence, which will be frequently used without explicit mention.
\begin{lemma}
If $p$ is a prime, $\alpha$ is a positive integer, then
\begin{align*}
\f{\alpha}^{p} &\equiv\f{p\alpha}\pmod{p},\\
\f{1}^{p^{\alpha}} &\equiv\f{p}^{p^{\alpha-1}}\pmod{p^{\alpha}}.
\end{align*}
\end{lemma}

\section{Congruences for $k$-colored partitions modulo 25}\label{sec:elementary pf}
\begin{proof}[Proof of Theorem \ref{thm:mod25}]
We will first prove Theorem \ref{thm:mod25} for $r=0$, then we will explain its connection with the remaining cases.

\textit{Case 1}. For $k=2$,
\begin{align*}
\sum_{n=0}^{\infty}p_{-2}(n)q^{n} &=\dfrac{1}{E_{1}^{2}}\\
 &=\dfrac{E_{25}^{10}}{E_{5}^{12}}\Bigg(\dfrac{1}{R(q^{5})^{8}}+\dfrac{2q}{R(q^{5})^{7}}
 +\dfrac{5q^{2}}{R(q^{5})^{6}}+\dfrac{10q^{3}}{R(q^{5})^{5}}+\dfrac{20q^{4}}{R(q^{5})^{4}}+\dfrac{16q^{5}}{R(q^{5})^{3}}\\
 &\quad+\dfrac{27q^{6}}{R(q^{5})^{2}}+\dfrac{20q^{7}}{R(q^{5})}+15q^{8}-20q^{9}R(q^{5})+27q^{10}R(q^{5})^{2}-16q^{11}R(q^{5})^{3}\\
 &\quad+20q^{12}R(q^{5})^{4}-10q^{13}R(q^{5})^{5}+5q^{14}R(q^{5})^{6}-2q^{15}R(q^{5})^{7}+q^{16}R(q^{5})^{8}\Bigg).
\end{align*}

Invoking \eqref{Rama cont frac 5}, we find that
\begin{align}
\sum_{n=0}^{\infty}p_{-2}(5n+3)q^{n} &=10\dfrac{E_{5}^{4}}{E_{1}^{6}}+125q\dfrac{E_{5}^{10}}{E_{1}^{12}}.\label{gf2:5n+3}
\end{align}

Furthermore, we get, (all the following congruences are modulo 25 in this section unless otherwise specified)
\begin{align*}
\sum_{n=0}^{\infty}p_{-2}(5n+3)q^{n} &\equiv10\dfrac{E_{5}^{3}}{E_{1}}=10E_{5}^{3}\sum_{n=0}^{\infty}p(n)q^{n}.
\end{align*}
It follows from \eqref{Rama mod 5} that
\begin{align*}
p_{-2}(25n+23)\equiv 0.
\end{align*}

\textit{Case 2}. For $k=17$, by \eqref{Hir dissection}, we get
\begin{align*}
\sum_{n=0}^{\infty}p_{-17}(n)q^{n} &=\dfrac{1}{E_{1}^{17}}=\dfrac{E_{1}^{8}}{E_{1}^{25}}\equiv\dfrac{E_{1}^{8}}{E_{5}^{5}}\\
 &=\dfrac{E_{25}^{8}}{E_{5}^{5}}\left(\dfrac{1}{R(q^{5})}-q-q^{2}R(q^{5})\right)^{8}\\
 &=\dfrac{E_{25}^{8}}{E_{5}^{5}}\Bigg(\dfrac{1}{R(q^{5})^{8}}-\dfrac{8q}{R(q^{5})^{7}}+\dfrac{20q^{2}}{R(q^{5})^{6}}-\dfrac{70q^{4}}{R(q^{5})^{4}}+\dfrac{56q^{5}}
 {R(q^{5})^{3}}+\dfrac{112q^{6}}{R(q^{5})^{2}}\\
 &\quad-\dfrac{120q^{7}}{R(q^{5})}-125q^{8}+120q^{9}R(q^{5})+112q^{10}R(q^{5})^{2}-56q^{11}R(q^{5})^{3}\\
 &\quad-70q^{12}R(q^{5})^{4}+20q^{14}R(q^{5})^{6}+8q^{15}R(q^{5})^{7}+q^{16}R(q^{5})^{8}\Bigg).
\end{align*}

It follows that
\begin{align*}
\sum_{n=0}^{\infty}p_{-17}(5n+3)q^{n}\equiv0.
\end{align*}

So
\begin{align}
p_{-17}(5n+3)\equiv0\label{p17:5n+3}
\end{align}
and, in particular,
\begin{align*}
p_{-17}(25n+8)\equiv0,
\end{align*}
which is what we wanted to prove. But note that we achieved a stronger result.

\textit{Case 3}. Similarly, for $k=11$, we have,
\begin{align*}
\sum_{n=0}^{\infty}p_{-11}(n)q^{n} &=\dfrac{1}{E_{1}^{11}}=\dfrac{E_{1}^{14}}{E_{1}^{25}}\equiv\dfrac{E_{1}^{14}}{E_{5}^{5}}\\
 &=\dfrac{E_{25}^{14}}{E_{5}^{5}}\Bigg(\dfrac{1}{R(q^{5})^{14}}-\dfrac{14q}{R(q^{5})^{13}}+\dfrac{77q^{2}}{R(q^{5})^{12}}-\dfrac{182q^{3}}{R(q^{5})^{11}}+\dfrac{910q^{5}}
 {R(q^{5})^{9}}-\dfrac{1365q^{6}}{R(q^{5})^{8}}\\
 &\quad-\dfrac{1430q^{7}}{R(q^{5})^{7}}+\dfrac{5005q^{8}}{R(q^{5})^{6}}-\dfrac{10010q^{10}}{R(q^{5})^{4}}+\dfrac{3640q^{11}}{R(q^{5})^{3}}+\dfrac{14105q^{12}}
 {R(q^{5})^{2}}-\dfrac{6930q^{13}}{R(q^{5})}\\
 &\quad-15625q^{14}+6930q^{15}R(q^{5})+14105q^{16}R(q^{5})^{2}-3640q^{17}R(q^{5})^{3}\\
 &\quad-10010q^{18}R(q^{5})^{4}+5005q^{20}R(q^{5})^{6}+1430q^{21}R(q^{5})^{7}-1365q^{22}R(q^{5})^{8}\\
 &\quad-910q^{23}R(q^{5})^{9}+182q^{25}R(q^{5})^{11}+77q^{26}R(q^{5})^{12}+14q^{27}R(q^{5})^{13}\\
 &\quad+q^{28}R(q^{5})^{14}\Bigg).
\end{align*}

It follows that
\begin{align}
p_{-11}(5n+4)\equiv0.\label{p11:5n+4}
\end{align}
In particular,
\begin{align*}
p_{-11}(25n+14)\equiv0.
\end{align*}

\textit{Case 4}. For $k=6$ and $7$, we need the following lemma.
\begin{lemma}
We have
\begin{align}
H\left(q^{i-5}\dfrac{E_{5}^{6i-1}}{E_{1}^{6i}}\right) &=\sum_{j=1}^{\infty}m(6i,i+j)q^{5j-5}\dfrac{E_{25}^{6j}}{E_{5}^{6j+1}},\label{H1}\\
H\left(q^{i-4}\dfrac{E_{5}^{6i}}{E_{1}^{6i+1}}\right) &=\sum_{j=1}^{\infty}m(6i+1,i+j)q^{5j-5}\dfrac{E_{25}^{6j-1}}{E_{5}^{6j}}.\label{H2}
\end{align}
\end{lemma}
\begin{proof}
It follows immediately from Lemma \ref{key lemma} and induction on $i$.
\end{proof}

Taking $i=1$ in \eqref{H1}, according to \eqref{recu formula} and the definition of $H$, then replacing $q^{5}$ by $q$, we obtain
\begin{align}
\sum_{n=0}^{\infty}p_{-6}(5n+4)q^{n} &=315\dfrac{E_{5}^{6}}{E_{1}^{12}}+52\times5^{4}q\dfrac{E_{5}^{12}}{E_{1}^{18}}+63\times5^{6}q^{2}
\dfrac{E_{5}^{18}}{E_{1}^{24}}\notag\\
 &\quad+6\times5^{9}q^{3}\dfrac{E_{5}^{24}}{E_{1}^{30}}+5^{11}q^{4}\dfrac{E_{5}^{30}}{E_{1}^{36}}\label{gf6:5n+4}
\end{align}
and
\begin{align*}
\sum_{n=0}^{\infty}p_{-6}(5n+4)q^{n} &\equiv315\dfrac{E_{5}^{4}}{E_{1}^{2}}=315E_{5}^{4}\sum_{n=0}^{\infty}p_{-2}(n)q^{n}.
\end{align*}
By \eqref{partition pairs mod 5}, we get
\begin{align*}
p_{-6}(25n+14)\equiv p_{-6}(25n+19)\equiv p_{-6}(25n+24)\equiv0.
\end{align*}

Putting $i=1$ in \eqref{H2} and by \eqref{recu formula}, we have
\begin{align}
\sum_{n=0}^{\infty}p_{-7}(5n+3)q^{n} &=140\dfrac{E_{5}^{5}}{E_{1}^{12}}+49\times5^{4}q\dfrac{E_{5}^{11}}{E_{1}^{18}}
+21\times5^{7}q^{2}\dfrac{E_{5}^{17}}{E_{1}^{24}}\notag\\
 &\quad+91\times5^{8}q^{3}\dfrac{E_{5}^{23}}{E_{1}^{30}}+7\times5^{11}q^{4}\dfrac{E_{5}^{29}}{E_{1}^{36}}
 +5^{13}q^{5}\dfrac{E_{5}^{35}}{E_{1}^{42}}\label{gf7:5n+3}
\end{align}
and
\begin{align*}
\sum_{n=0}^{\infty}p_{-7}(5n+3)q^{n} &\equiv140\dfrac{E_{5}^{3}}{E_{1}^{2}}=140E_{5}^{3}\sum_{n=0}^{\infty}p_{-2}(n)q^{n}.
\end{align*}
Similarly, we obtain
\begin{align*}
p_{-7}(25n+13)\equiv p_{-7}(25n+18)\equiv p_{-7}(25n+23)\equiv0.
\end{align*}

\textit{Case 5}. For $r\geq1$ and $k\in\{1,2,6,7,11,17\}$, assume that $k=5s+t$ $(1\leq t\leq4)$. We consider the following two cases:
\begin{enumerate}[1)]
\item $k\in\{11, 17\}$.
\begin{align*}
\sum_{n=0}^{\infty}p_{-(25r+k)}(n)q^{n} &=\dfrac{1}{E_{1}^{25r+k}}\equiv\dfrac{1}{E_{1}^{k}E_{5}^{5r}}=\dfrac{1}{E_{5}^{5r}}\sum_{n=0}^{\infty}p_{-k}(n)q^{n}.
\end{align*}
Hence
\begin{align*}
\sum_{n=0}^{\infty}p_{-(25r+k)}(5n+5-t)q^{n}\equiv\dfrac{1}{E_{5}^{5r}}\sum_{n=0}^{\infty}p_{-k}(5n+5-t)q^{n}.
\end{align*}

The Eqs. \eqref{p17:5n+3} and \eqref{p11:5n+4} imply
\begin{align*}
p_{-k}\left(5n+5-t\right)\equiv0.
\end{align*}
It follows immediately that
\begin{align*}
p_{-(25r+k)}\left(5n+5-t\right)\equiv0,
\end{align*}
since all terms in the factor $1/E_{5}^{5r}$ are of the form $q^{5i}$.

\item $k\in\{1, 2, 6, 7\}$.
\begin{align*}
\sum_{n=0}^{\infty}p_{-(25r+k)}(n)q^{n} &=\dfrac{1}{E_{1}^{25r+5s+t}}=\zeta^{-25r-5s-t}\dfrac{1}{q^{25r+5s+t}E_{25}^{25r+5s+t}}.
\end{align*}
Picking out those terms of the form $q^{5n+5-t}$ and applying Lemma \ref{key lemma}, we find that
\begin{align*}
\sum_{n=0}^{\infty}p_{-(25r+k)}(5n+5-t)q^{n}=\sum_{5r+s+1}^{25r+5s+t}m(25r+5s+t,h)\dfrac{q^{h-5r-s-1}}{{E_{1}^{6h}}E_{5}^{25r+5s+t-6h}}.
\end{align*}
According to Lemma \ref{lemma:estima}, we know that $\pi_{5}(m(25r+5s+t,h))\geq2$ if $h\geq5r+s+2$, and $\pi_{5}(m\left(25r+5s+t,5r+s+1\right))\geq1$ if $t=1$ or $2$. Therefore
\begin{align*}
\sum_{n=0}^{\infty}p_{-(25r+k)}(5n+5-t)q^{n} &\equiv m\left(25r+5s+t,5r+s+1\right)\dfrac{E_{5}^{5r+s-t+6}}{E_{1}^{30r+6s+6}}\\
 &\equiv m\left(25r+5s+t,5r+s+1\right)\dfrac{E_{5}^{5-r-t}}{E_{1}^{s+1}}.
\end{align*}
\begin{itemize}
\item If $k\in\{1, 2\}$, then $s=0$. By \eqref{Rama mod 5}, we obtain
\begin{align*}
p_{-(25r+k)}(25n+25-t)\equiv0.
\end{align*}

\item If $k\in\{6, 7\}$, then $s=1$. Upon \eqref{partition pairs mod 5}, we get
\begin{align*}
p_{-(25r+k)}(25n+15-t)\equiv p_{-(25r+k)}(25n+20-t)\equiv p_{-(25r+k)}(25n+25-t)\equiv0.
\end{align*}
\end{itemize}

\end{enumerate}
This proves \eqref{mod 25 family}.
\end{proof}

As an immediate consequence, we obtain the following corollary.
\begin{corollary}
For all non-negative integers $r$ and $n$, we have
\begin{align*}
p_{25r+16}(25n+19) &\equiv p_{25r+16}(25n+24) \equiv 0,\\
p_{25r+21}(25n+9) \equiv p_{25r+21}(25n+14) &\equiv p_{25r+21}(25n+19) \equiv p_{25r+21}(25n+24) \equiv 0,\\
p_{25r+22}(25n+8) \equiv p_{25r+22}(25n+13) &\equiv p_{25r+22}(25n+18) \equiv p_{25r+22}(25n+23) \equiv 0.
\end{align*}
\end{corollary}

\section{Congruences for $k$-colored partitions modulo powers of 5}\label{mod5 power}
\subsection{Congruences for $p_{-2}(n)$ modulo powers of 5}
In this subsection, we will prove \eqref{5 power cong:odd} and \eqref{5 power cong:even}. To present the generating functions for the sequence in \eqref{5 power cong:odd} and \eqref{5 power cong:even}, we need to define another infinite matrix of natural numbers $\{a(j,k)\}_{j,k\geq1}$ by
\begin{enumerate}[1)]
\item $a(1,1)=10$, $a(1,2)=125$, and $a(1,k)=0$ for $k\geq3$.
\item
\begin{align*}
a(j+1,k)=
\begin{cases}
\sum_{i=1}^{\infty}a(j,i)m(6i,i+k)\quad &\textrm{if}~j~\textrm{is~odd}, \cr \sum_{i=1}^{\infty}a(j,i)m(6i+2,i+k)\quad &\textrm{if}~j~\textrm{is~even}.
\end{cases}
\end{align*}
\end{enumerate}
According to Lemma \ref{zero vaules}, the summation in $2)$ is indeed finite.

To prove \eqref{5 power cong:odd}--\eqref{5 power cong:even}, we need the following key theorems and lemmas.
\begin{theorem}\label{theorem:mod power of 5:even-odd}
For any positive integer $j$, we have
\begin{align}
\sum_{n=0}^{\infty}p_{-2}\left(5^{2j-1}n+\dfrac{7\times5^{2j-1}+1}{12}\right)q^{n} &=\sum_{l=1}^{\infty}a(2j-1,l)q^{l-1}\dfrac{E_{5}^{6l-2}}{E_{1}^{6l}},
\label{gf:mod power 5:odd}\\
\sum_{n=0}^{\infty}p_{-2}\left(5^{2j}n+\dfrac{11\times5^{2j}+1}{12}\right)q^{n} &=\sum_{l=1}^{\infty}a(2j,l)q^{l-1}\dfrac{E_{5}^{6l}}{E_{1}^{6l+2}}.
\label{gf:mod power 5:even}
\end{align}
\end{theorem}

\begin{proof}
We proceed by induction on $j$. By \eqref{gf2:5n+3}, we know that \eqref{gf:mod power 5:odd} holds for $j=1$. Assume that \eqref{gf:mod power 5:odd} is true for some positive integer $j\geq1$. We restate it as
\begin{align*}
\sum_{n=0}^{\infty}p_{-2}\left(5^{2j-1}n+\dfrac{7\times5^{2j-1}+1}{12}\right)q^{n}=\dfrac{1}{qE_{5}^{2}}\sum_{l=1}^{\infty}a(2j-1,l)T^{l}\zeta^{-6l}.
\end{align*}
Picking out those terms of the form $q^{5n+4}$ and applying Lemma \ref{key lemma}, we obtain
\begin{align}
 &\sum_{n=0}^{\infty}p_{-2}\left(5^{2j-1}(5n+4)+\dfrac{7\times5^{2j-1}+1}{12}\right)q^{5n+4}\notag\\
 =&\dfrac{1}{qE_{5}^{2}}\sum_{l=1}^{\infty}a(2j-1,l)T^{l}\left(\sum_{k=1}^{\infty}m(6l,k)T^{-k}\right).\label{key equation}
\end{align}
According to Lemma \ref{zero vaules}, we know that $m(6l,k)\neq0$ implies $k\geq l+1$. Now \eqref{key equation} implies
\begin{align*}
 &\sum_{n=0}^{\infty}p_{-2}\left(5^{2j}n+\dfrac{11\times5^{2j}+1}{12}\right)q^{n}\\
 =&\dfrac{1}{qE_{1}^{2}}\sum_{l=1}^{\infty}\sum_{k=l+1}^{\infty}a(2j-1,l)m(6l,k)\left(q\dfrac{E_{5}^{6}}{E_{1}^{6}}\right)^{k-l}
 \quad(\textrm{replace}~k~\textrm{by}~k+l)\\
 =&\dfrac{1}{qE_{1}^{2}}\sum_{k=1}^{\infty}\sum_{l=1}^{\infty}a(2j-1,l)m(6l,k+l)\left(q\dfrac{E_{5}^{6}}{E_{1}^{6}}\right)^{k}\\
 =&\sum_{k=1}^{\infty}a(2j,k)q^{k-1}\dfrac{E_{5}^{6k}}{E_{1}^{6k+2}}.
\end{align*}
This implies that \eqref{gf:mod power 5:even} holds for $j$. Similarly, we rewrite \eqref{gf:mod power 5:even} as
\begin{align*}
\sum_{n=0}^{\infty}p_{-2}\left(5^{2j}n+\dfrac{11\times5^{2j}+1}{12}\right)q^{n}=\dfrac{1}{q^{3}E_{25}^{2}}\sum_{l=1}^{\infty}a(2j,l)T^{l}\zeta^{-(6l+2)}.
\end{align*}
Taking out those terms of the form $q^{5n+2}$ and applying Lemma \ref{key lemma}, we find
\begin{align}
 &\sum_{n=0}^{\infty}p_{-2}\left(5^{2j}(5n+2)+\dfrac{11\times5^{2j}+1}{12}\right)q^{5n+2}\notag\\
 =&\dfrac{1}{q^{3}E_{25}^{2}}\sum_{l=1}^{\infty}a(2j,l)T^{l}\left(\sum_{k=1}^{\infty}m(6l+2,k)T^{-k}\right).\label{key equation2}
\end{align}
By Lemma \ref{zero vaules}, we know that $m(6l+2,k)\neq0$ implies $k\geq l+1$. Now \eqref{key equation2} implies
\begin{align*}
 &\sum_{n=0}^{\infty}p_{-2}\left(5^{2j+1}n+\dfrac{7\times5^{2j+1}+1}{12}\right)q^{n}\\
 =&\dfrac{1}{qE_{5}^{2}}\sum_{l=1}^{\infty}\sum_{k=l+1}^{\infty}a(2j,l)m(6l+2,k)\left(q\dfrac{E_{5}^{6}}{E_{1}^{6}}\right)^{k-l}
 \quad(\textrm{replace}~k~\textrm{by}~k+l)\\
 =&\dfrac{1}{qE_{5}^{2}}\sum_{k=1}^{\infty}\sum_{l=1}^{\infty}a(2j,l)m(6l+2,k+l)\left(q\dfrac{E_{5}^{6}}{E_{1}^{6}}\right)^{k}\\
 =&\sum_{k=1}^{\infty}a(2j+1,k)q^{k-1}\dfrac{E_{5}^{6k-2}}{E_{1}^{6k}}.
\end{align*}
This implies that \eqref{gf:mod power 5:odd} holds for $j+1$. This finishes the proof by induction.
\end{proof}

For any positive integer $n$, let $\pi_{5}(n)$ enumerate the highest power of 5 that divides $n$. For convention, we define $\pi_{5}(0)=+\infty$. To prove \eqref{5 power cong 1}--\eqref{5 power cong 3}, we need the following lemma to estimate $\pi_{5}\left(a(j,k)\right)$.
\begin{lemma}[Lemma 4.1, \cite{HH1981}]\label{lemma:estima}
For any positive integers $j\geq1$, we have
\begin{align*}
\pi_{5}(m(i,j))\geq\left\lfloor\dfrac{5j-i-1}{2}\right\rfloor.
\end{align*}
\end{lemma}

\begin{lemma}\label{estima:lemma1}
For any positive integers $j\geq1$ and $k\geq1$, we have
\begin{align}
\pi_{5}(a(2j-1,k)) &\geq j+\left\lfloor\dfrac{5k-5}{2}\right\rfloor,\label{estima:odd}\\
\pi_{5}(a(2j,k)) &\geq j+\left\lfloor\dfrac{5k-3}{2}\right\rfloor.\label{estima:even}
\end{align}
\end{lemma}
\begin{proof}
It is easy to see that \eqref{estima:odd} holds for $j=1$. Assume \eqref{estima:odd} is true for $j\geq1$. By definition of $\pi_{5}$ and Lemma \ref{lemma:estima}, we get
\begin{align}
\pi_{5}(a(2j,k)) &=\pi_{5}\left(\sum_{i=1}^{\infty}a(2j-1,i)m(6i,k+i)\right)\notag\\
 &\geq\min_{i\geq1}\left(\pi_{5}(a(2j-1,i))+\pi_{5}(m(6i,k+i))\right)\notag\\
 &\geq\min_{i\geq1}\left(j+\left\lfloor\dfrac{5i-5}{2}\right\rfloor+\left\lfloor\dfrac{5k-i-1}{2}\right\rfloor\right).\label{inequlity}
\end{align}
Let
\begin{align*}
g(i,k)=\left\lfloor\dfrac{5i-5}{2}\right\rfloor+\left\lfloor\dfrac{5k-i-1}{2}\right\rfloor.
\end{align*}

Notice that for fixed $k$, if we increase $i$ by 1, $\left\lfloor\dfrac{5i-5}{2}\right\rfloor$ increases by at least 2, but $\left\lfloor\dfrac{5k-i-1}{2}\right\rfloor$ decreases by at most 1. Hence $g(i+1,k)\geq g(i,k)+1$. Thus we obtain
\begin{align*}
g(i,k)\geq g(1,k)=\left\lfloor\dfrac{5k-2}{2}\right\rfloor\geq\left\lfloor\dfrac{5k-3}{2}\right\rfloor.
\end{align*}
Thus, we derive form \eqref{inequlity} that
\begin{align*}
\pi_{5}(a(2j,k))\geq j+\left\lfloor\dfrac{5k-3}{2}\right\rfloor.
\end{align*}
This proves that \eqref{estima:even} holds for $j$.

Similarly, we find
\begin{align}
\pi_{5}(a(2j+1,k)) &=\pi_{5}\left(\sum_{i=1}^{\infty}a(2j,i)m(6i+2,k+i)\right)\notag\\
 &\geq \min_{i\geq1}\left(j+\left\lfloor\dfrac{5i-3}{2}\right\rfloor+\left\lfloor\dfrac{5k-i-3}{2}\right\rfloor\right)\notag\\
 &\geq j+1+\left\lfloor\dfrac{5k-5}{2}\right\rfloor.\label{inequlity2}
\end{align}
Here the last equality in \eqref{inequlity2} because the minimal value occurs at $i=1$. This shows that \eqref{estima:odd} holds for $j+1$. The proof is completed by induction.
\end{proof}

The congruence \eqref{5 power cong:odd} follows from \eqref{gf:mod power 5:odd} together with \eqref{estima:odd}, and the congruence \eqref{5 power cong:even} follows from \eqref{gf:mod power 5:even} together with \eqref{estima:even}.

\subsection{Congruences for $p_{-6}(n)$ modulo powers of 5}
Now, we apply the same method to investigate the arithmetic properties of $p_{-6}(n)$. Define
\begin{enumerate}[1)]
\item $b(1,1)=315$, $a(1,2)=52\times5^{4}$, $b(1,3)=63\times5^{6}$, $b(1,4)=6\times5^{9}$, $b(1,5)=5^{11}$ and $b(1,k)=0$ for $k\geq6$.
\item
\begin{align*}
b(j+1,k)=\sum_{i=1}^{\infty}b(j,i)m(6i+6,k+i+1),\quad j\geq1, \quad k\geq1.
\end{align*}
\end{enumerate}

\begin{theorem}\label{theorem:genral gf1}
For any positive integer $j$, we have
\begin{align}
\sum_{n=0}^{\infty}p_{-6}\left(5^{j}n+\dfrac{3\times5^{j}+1}{4}\right)q^{n}=\sum_{l=1}^{\infty}b(j,l)q^{l-1}\dfrac{E_{5}^{6l}}{E_{1}^{6l+6}}.\label{gf:general1}
\end{align}
\end{theorem}
\begin{proof}
We proceed by induction on $j$. According to \eqref{gf6:5n+4}, we know that \eqref{gf:general1} is true for $j=1$. Assume that \eqref{gf:general1} holds for some natural number $j\geq1$. We rewrite it as
\begin{align*}
\sum_{n=0}^{\infty}p_{-6}\left(5^{j}n+\dfrac{3\times5^{j}+1}{4}\right)q^{n}=\dfrac{1}{q^{7}E_{25}^{6}}\sum_{l=1}^{\infty}b(j,l)T^{l}\zeta^{-(6l+6)}.
\end{align*}
Picking out those terms of the form $q^{5n+3}$ and applying Lemma \ref{key lemma}, we find that
\begin{align*}
\sum_{n=0}^{\infty}p_{-6}\left(5^{j}(5n+3)+\dfrac{3\times5^{j}+1}{4}\right)q^{5n+3} &=\dfrac{1}{q^{7}E_{25}^{6}}\sum_{l=1}^{\infty}b(j,l)T^{l}H\left(\zeta^{-(6l+6)}\right)\\
&=\dfrac{1}{q^{7}E_{25}^{6}}\sum_{l=1}^{\infty}\sum_{k=1}^{\infty}b(j,l)m(6l+6,k)T^{l-k}.
\end{align*}
By Lemma \ref{zero vaules}, we know that $m(6l+6,k)\neq0$ implies $k\geq l+2$. Dividing both sides by $q^{3}$ and replacing $q^{5}$ by $q$, we get
\begin{align*}
 &\sum_{n=0}^{\infty}p_{-6}\left(5^{j+1}n+\dfrac{3\times5^{j+1}+1}{4}\right)q^{n}\\
 =&\dfrac{1}{q^{2}E_{5}^{6}}\sum_{l=1}^{\infty}\sum_{k=l+2}^{\infty}b(j,l)m(6l+6,k)\left(q\dfrac{E_{5}^{6}}{E_{1}^{6}}\right)^{k-l}
 \quad(\textrm{replace}~k~\textrm{by}~k+l+1)\\
 =&\dfrac{1}{q^{2}E_{5}^{6}}\sum_{l=1}^{\infty}\sum_{k=1}^{\infty}b(j,l)m(6l+6,k+l+1)\left(q\dfrac{E_{5}^{6}}{E_{1}^{6}}\right)^{k+1}\\
 =&\sum_{k=1}^{\infty}b(j+1,k)q^{k-1}\dfrac{E_{5}^{6l}}{E_{1}^{6l+6}}.
\end{align*}
This implies that \eqref{gf:general1} holds for $j+1$. Thus we complete the proof by induction.
\end{proof}

\begin{lemma}
For any positive integers $j\geq1$ and $k\geq1$, we have
\begin{align}
\pi_{5}\left(b(j,k)\right)\geq j+\left\lfloor\dfrac{5k-5}{2}\right\rfloor.\label{estima1}
\end{align}
\end{lemma}
\begin{proof}
It is obvious that \eqref{estima1} holds for $j=1$. Assume \eqref{estima1} holds for some $j\geq1$, then by Lemma \ref{lemma:estima} we get
\begin{align*}
\pi\left(b(j+1,k)\right) &=\pi\left(\sum_{i=1}^{\infty}b(j,i)m(6i+6,k+i+1)\right)\\
 &\geq\min_{i\geq1}\bigg(\pi\big(b(j,i)\big)+\pi\big(m(6i+6,k+i+1)\big)\bigg)\\
 &\geq j+\left\lfloor\dfrac{5i-5}{2}\right\rfloor+\left\lfloor\dfrac{5k-i-2}{2}\right\rfloor\\
 &\geq j+1+\left\lfloor\dfrac{5k-5}{2}\right\rfloor.
\end{align*}
Hence, \eqref{estima1} holds for $j+1$ and therefore for all integers $j\geq1$ by induction.
\end{proof}

It follows easily from \eqref{estima1} that
\begin{align*}
\pi_{5}\left(b(j,k)\right)\geq j+\left\lfloor\dfrac{5k-5}{2}\right\rfloor\geq j+2
\end{align*}
for $k\geq2$.

By \eqref{gf:general1}, we get, modulo $5^{j+1}$,
\begin{align}
\sum_{n=0}^{\infty}p_{-6}\left(5^{j}n+\dfrac{3\times5^{j}+1}{4}\right)q^{n} &\equiv b(j,1)\dfrac{E_{5}^{6}}{E_{1}^{12}}\equiv b(j,1)E_{5}^{4}\sum_{n=0}^{\infty}p_{-2}(n)q^{n}.\label{modulo congruence}
\end{align}

Since $\pi_{5}(b(j,1))\geq j$, the congruences \eqref{5 power cong 2} and \eqref{5 power cong 3} follow from \eqref{modulo congruence} together with \eqref{partition pairs mod 5}.

\subsection{Congruences for $p_{-7}(n)$ modulo powers of 5}
This case is similar to the case $k=2$, we present here the main results and omit their proofs. Let
\begin{enumerate}[1)]
\item $c(1,1)=140$, $a(1,2)=49\times5^{4}$, $c(1,3)=21\times5^{7}$, $c(1,4)=91\times5^{8}$, $c(1,5)=7\times5^{11}$, $c(1,6)=5^{13}$ and $c(1,k)=0$ for $k\geq7$.
\item
\begin{align*}
c(j+1,k)=
\begin{cases}
\sum_{i=1}^{\infty}c(j,i)m(6i+6,i+k+1)\quad &\textrm{if}~j~\textrm{is~odd}, \cr \sum_{i=1}^{\infty}c(j,i)m(6i+7,i+k+1)\quad &\textrm{if}~j~\textrm{is~even}.
\end{cases}
\end{align*}
\end{enumerate}

\begin{theorem}
For any positive integer $j$, we have
\begin{align}
\sum_{n=0}^{\infty}p_{-7}\left(5^{2j-1}+\dfrac{13\times5^{2j-1}+7}{24}\right)q^{n} &=\sum_{l=1}^{\infty}c(2j-1,l)q^{l-1}\dfrac{E_{5}^{6l-1}}{E_{1}^{6l+6}},\label{gf7:odd}\\
\sum_{n=0}^{\infty}p_{-7}\left(5^{2j}+\dfrac{17\times5^{2j}+7}{24}\right)q^{n} &=\sum_{l=1}^{\infty}c(2j,l)q^{l-1}\dfrac{E_{5}^{6l}}{E_{1}^{6l+7}}.\notag
\end{align}
\end{theorem}

\begin{lemma}
For any positive integers $j\geq1$ and $k\geq1$, we have
\begin{align}
\pi_{5}(c(2j-1,k)) &\geq j+\left\lfloor\dfrac{5k-5}{2}\right\rfloor,\label{estima2:odd}\\
\pi_{5}(c(2j,k)) &\geq j+\left\lfloor\dfrac{5k-3}{2}\right\rfloor.\notag
\end{align}
\end{lemma}

It follows immediately from \eqref{estima2:odd} that
\begin{align*}
\pi_{5}(c(2j-1,k)) &\geq j+\left\lfloor\dfrac{5k-5}{2}\right\rfloor\geq j+2
\end{align*}
for $k\geq2$.

According to \eqref{gf7:odd}, we have, modulo $5^{j+1}$,
\begin{align}
\sum_{n=0}^{\infty}p_{-7}\left(5^{2j-1}+\dfrac{13\times5^{2j-1}+7}{24}\right)q^{n} &\equiv c(2j-1,1)\dfrac{E_{5}^{5}}{E_{1}^{12}}\equiv c(2j-1,1)E_{5}^{3}\sum_{n=0}^{\infty}p_{-2}(n)q^{n}.\label{modulo gf 7}
\end{align}

Eqs. \eqref{5 power cong:odd 5n+2}--\eqref{5 power cong:5n+4} are immediate consequence of \eqref{estima2:odd}, \eqref{modulo gf 7} and \eqref{partition pairs mod 5}.

\section{Final Remarks}
A number of congruences satisfied by $k$-colored partitions have been found (see \cite{Atk1968, Gor1981, KO1992, Andr2008, FO2012, LW2017}, to name a few). For example, Atkin \cite{Atk1968} proved the following infinite families of congruences modulo powers of prime.
\begin{theorem}[Theorem 1.1, \cite{Atk1968}]
Suppose $k>0$ and $q=2,3,5,7$ or $13$. If $24n\equiv k\pmod{q^{r}}$, then $p_{-k}(n)\equiv0\pmod{q^{\frac{1}{2}\alpha r+\epsilon}}$, where $\epsilon=\epsilon(q,k)=O(\log k)$, and where $\alpha$ depends on $q$ and the residue of $k$ modulo $24$ according to a certain table.
\end{theorem}

Applying the operator $H$, we also obtain some infinite families of congruences modulo powers of 5 for $k=11$ by following the same line of proving Theorems \ref{beau power} and \ref{thm:cong high power}. However, for $k=17$, it seems that there do not exist congruences modulo powers of 5 similar to the types of \eqref{5 power cong:odd}--\eqref{5 power cong 3}. Interestingly, Atkin's results also assert that $\alpha=0$ when $k\equiv17\pmod{24}$ and $q=5$.

\section*{Acknowledgement}
I am indebted to Ernest X. W. Xia, Michael D. Hirschhorn, Shishuo Fu and Shane Chern  for their helpful comments and suggestions that have improved this paper to a great extent. I would like to thank the referee who read the original carefully, picked up a number of typos and made some helpful comments. This work was supported by the National Natural Science Foundation of China (No.~11501061).

\end{document}